\documentclass{amsart}

\usepackage[colorlinks]{hyperref}
\usepackage{amssymb}
\usepackage[all]{xy}
\usepackage{booktabs}

\newcommand{\AT}{ArtinTate2009}
\newcommand{\Gras}{Gras2003}
\newcommand{\HA}{Hasse2002}
\newcommand{\KA}{Kaplansky1954}
\newcommand{\KU}{Kubota1957}
\newcommand{\MA}{MacLane1995}
\newcommand{\NSW}{NeukirchSchmidtEtAl2008}
\newcommand{\NeuOne}{Neukirch1969}
\newcommand{\NeuTwo}{Neukirch1977}
\newcommand{\ONa}{Onabe1976}

\newcommand{\WA}{Watkins2004}

\newtheorem{theorem}{Theorem}[section]

\newtheorem{corollary}[theorem]{Corollary}
\newtheorem{lemma}[theorem]{Lemma}
\newtheorem{conjecture}[theorem]{Conjecture}

\theoremstyle{definition}

\newtheorem{algorithm}[theorem]{Algorithm}

\newtheorem*{remark*}{Remark}

\DeclareMathOperator{\characteristic}{char}
\DeclareMathOperator{\Cl}{Cl}
\DeclareMathOperator{\Ext}{Ext}
\DeclareMathOperator{\Gal}{Gal}
\DeclareMathOperator{\Hom}{Hom}
\DeclareMathOperator{\ord}{ord}

\newcommand{\CC}{{\mathbf{C}}}
\newcommand{\FF}{{\mathbf{F}}}
\newcommand{\Zee}{{\mathbf{Z}}}
\newcommand{\Zhat}{{\widehat{\Zee}}}
\newcommand{\Khat}{{\widehat{K}}}
\newcommand{\Que}{\mathbf{Q}}
\newcommand{\Ocal}{{\mathcal{O}}}
\newcommand{\Ocalhat}{{\widehat{\mathcal{O}}}}
\newcommand{\Ohat}{{\widehat{\mathcal{O}}}}

\newcommand\gotha{{\mathfrak{a}}}
\newcommand\gothp{{\mathfrak{p}}}
\newcommand\gothq{{\mathfrak{q}}}

\newcommand{\tto}{\longrightarrow}
\newcommand{\congr}{\equiv}
\newcommand{\iso}{\cong}
\newcommand{\ab}{\text{ab}}

\newcommand{\mapright}[1]{\mathop{\longrightarrow}\limits^{#1}}
\def\isar{\ \mapright{\sim}\ }

\renewcommand{\mod}{\bmod}

\newlength{\algindent}
\newlength{\alglabel}
\newcounter{stepcount}

\newenvironment{alglist}
{\quad\begin{list}{\arabic{stepcount}.}{\leftmargin=\algindent\labelwidth=\algindent\itemsep=\smallskipamount\usecounter{stepcount}}}
{\end{list}}

\newcommand{\algin}{\item[\emph{Input}:]}
\newcommand{\algout}{\item[\emph{Output}:]}

\begin{document}

\title[Imaginary quadratic fields with isomorphic abelian Galois groups]
{Imaginary quadratic fields with \\ isomorphic abelian Galois groups}

\author{Athanasios Angelakis}
\address{Mathematisch Instituut,
         Universiteit Leiden,
         Postbus 9512,
         2300 RA Leiden, The Netherlands}
\email{aangelakis@math.leidenuniv.nl}

\author{Peter Stevenhagen}
\address{Mathematisch Instituut,
         Universiteit Leiden,
         Postbus 9512,
         2300 RA Leiden, The Netherlands}
\email{psh@math.leidenuniv.nl}



\begin{abstract}
In 1976, Onabe discovered that, in contrast to the Neukirch-Uchida
results that were proved around the same time, a number field $K$ is not
completely characterized by its absolute abelian Galois group $A_K$.
The first examples of nonisomorphic $K$ having isomorphic $A_K$
were obtained on the basis of a classification by Kubota of idele class
character groups in terms of their infinite families of Ulm invariants,
and did not yield a description of~$A_K$.
In this paper, we provide a direct `computation' of the profinite 
group $A_K$ for imaginary quadratic $K$, and use it to obtain
\emph{many} different $K$ that all have 
the \emph{same} \emph{minimal} absolute abelian Galois group.
\end{abstract}

\maketitle

\section{Introduction}
\label{S:intro}
The absolute Galois group $G_K$ of a number field $K$ is a
large profinite group that we cannot currently describe in very
precise terms.
This makes it impossible to answer fundamental questions on $G_K$,
such as the inverse Galois problem over~$K$.
Still, Neukirch \cite{\NeuOne} proved that normal number fields
are completely characterized by their absolute Galois groups:
If $G_{K_1}$ and $G_{K_2}$ are isomorphic as topological groups, then
$K_1$ and $K_2$ are isomorphic number fields.
The result was refined by Ikeda, Iwasawa, and Uchida
(\cite{\NeuTwo}, \cite[Chapter XII, \S2]{\NSW}),
who disposed of the restriction to normal number fields, and showed
that every topological isomorphism $G_{K_1}\isar G_{K_2}$ is actually induced
by an inner automorphism of $G_\Que$.
The same statements hold if all absolute Galois groups are replaced
by their maximal \emph{prosolvable} quotients.

It was discovered by Onabe \cite{\ONa} that the situation changes
if one moves a further step down from $G_K$, to its maximal \emph{abelian}
quotient $A_K=G_K/\overline{[G_K, G_K]}$, which is
the Galois group $A_K=\Gal(K^\ab/K)$ of the maximal abelian
extension $K^\ab$ of $K$.
Even though the Hilbert problem of explicitly generating $K^\ab$ 
for general number fields $K$ is still open after more than a century,
the group $A_K$ can be described by class field theory, as
a quotient of the idele class group of $K$.

Kubota \cite{\KU} studied the group $X_K$ of continuous characters on $A_K$,
and expressed the structure of the $p$-primary parts of
this countable abelian torsion group 
in terms of an infinite number of so-called \emph{Ulm invariants}.
It had been shown by Kaplansky \cite[Theorem 14]{\KA} that such
invariants determine the isomorphism type of a countable reduced
abelian torsion group.
Onabe computed the Ulm invariants of $X_K$ explicitly for a number of small
imaginary quadratic fields $K$, and concluded from this that there exist
nonisomorphic imaginary quadratic fields $K$ and $K'$ for which the
absolute abelian Galois groups $A_K$ and $A_{K'}$ are isomorphic as
profinite groups.
This may even happen in cases where $K$ and $K'$ have different 
class numbers, but the explicit example $K=\Que(\sqrt{-2})$,
$K'=\Que(\sqrt{-5})$ of this that occurs in Onabe's main theorem
\cite[Theorem 2]{\ONa} is incorrect.
This is because the value of the finite Ulm invariants in
\cite[Theorem 4]{\KU} is incorrect for the prime $2$ in case the
ground field is a special number field in the sense of our 
Lemma \ref{lemma:T_K}.
As it happens, $\Que(\sqrt{-5})$ and the exceptional field
$\Que(\sqrt{-2})$ do have different Ulm invariants at 2.
The nature of Kubota's error is similar to an error in Grunwald's theorem
that was corrected by a theorem of Wang occurring in Kubota's paper
\cite[Theorem 1]{\KU}.
It is related to the noncyclic nature of the 2-power cyclotomic 
extension $\Que\subset \Que(\zeta_{2^\infty})$.

In this paper, we obtain Onabe's corrected results by a direct
class field theoretic approach that completely avoids Kubota's dualization
and the machinery of Ulm invariants.
We show that the imaginary quadratic fields $K\ne \Que(\sqrt{-2})$
that are said to be of `type~A' in~\cite{\ONa} share a \emph{minimal}
absolute abelian Galois group that can be described completely 
explicitly as
$$
A_K=\Zhat^2 \times \prod_{n\ge1} \Zee/n\Zee.
$$
The numerical data that we present suggest that these fields are
in fact very common among imaginary quadratic fields: More than 97\% of 
the 2356 fields of odd prime class number $h_K=p<100$ are of this nature.
We believe (Conjecture \ref{conj:infmany})
that there are actually \emph{infinitely many}
$K$ for which $A_K$ is the minimal group above.
Our belief is supported by certain reasonable assumptions on
the average splitting behavior of exact sequences of abelian groups, and
these assumptions are tested numerically in the final section
of the paper.

\section{Galois groups as \texorpdfstring{$\Zhat$}{Zhat}-modules}
\label{S:Zhatmodules}
The profinite abelian Galois groups that we study in this paper
naturally come with a topology for which the identity has a basis of open
neighborhoods that are open subgroups of finite index.
This implies that they are not simply $\Zee$-modules, but that
the exponentiation in these groups with ordinary integers extends to
exponentiation with elements of the profinite completion
$\Zhat=\lim_{\leftarrow n} \Zee/n\Zee$ of $\Zee$.
By the Chinese remainder theorem, we have a decomposition
of the profinite ring $\Zhat=\prod_p\Zee_p$ into a product
of rings of $p$-adic integers, with $p$ ranging over all primes.
As $\Zhat$-modules, our Galois groups decompose correspondingly
as a product of pro-$p$-groups.

It is instructive to look first at the $\Zhat$-module structure of
the absolute abelian Galois group $A_\Que$ of $\Que$,
which we know very explicitly by the Kronecker-Weber theorem.
This theorem states that $\Que^\ab$ is the maximal cyclotomic
extension of~$\Que$, and that an element $\sigma\in A_\Que$ acts on the roots
of unity that generate $\Que^\ab$ by exponentiation.
More precisely, we have $\sigma(\zeta)=\zeta^u$ for all roots of unity,
with $u$ a uniquely defined element $u$ in the unit group $\Zhat^*$ of the
ring $\Zhat$.
This yields the well-known isomorphism
$A_\Que=\Gal(\Que^\ab/\Que)\iso \Zhat^*=\prod_p\Zee_p^*$.

For odd $p$, the group $\Zee_p^*$ consists of a finite torsion subgroup $T_p$
of $(p-1)$-st roots of unity, and we have an isomorphism
$$
\Zee_p^* = T_p \times (1+p\Zee_p) \iso T_p \times \Zee_p
$$
because $1+p\Zee_p$ is a free $\Zee_p$-module generated by $1+p$.
For $p=2$ the same is true with $T_2=\{\pm1\}$ and $1+4\Zee_2$
the free $\Zee_2$-module generated by $1+4=5$.
Taking the product over all $p$, we obtain
\begin{equation}
\label{eq:A_Q}
A_\Que\iso T_\Que \times \Zhat,
\end{equation}
with $T_\Que=\prod_p T_p$ the product of the torsion subgroups 
$T_p\subset\Que_p^*$ of the multiplicative groups of the
completions $\Que_p$ of~$\Que$.
More canonically, $T_\Que$ is the \emph{closure} of the torsion
subgroup of $A_\Que=\Gal(\Que^\ab/\Que)$,
and $A_\Que/T_\Que$ is a free $\Zhat$-module of rank~1.
The invariant field of $T_\Que$ inside $\Que^\ab$ is the unique
$\Zhat$-extension of $\Que$.

Even though it looks at first sight as if the isomorphism type of
$T_\Que$ depends on the properties of prime numbers,
one should realize that in an infinite product of
finite cyclic groups, the Chinese remainder theorem allows us to
rearrange factors in many different ways.
One has for instance a noncanonical isomorphism
\begin{equation}
\label{eq:T_Q}
T_\Que = \prod_p T_p \iso \prod_{n\ge 1} \Zee/n\Zee,
\end{equation}
as both of these products, when written as a countable product of
cyclic groups of prime power order, have an infinite number of
factors $\Zee/\ell^k\Zee$ for each prime power~$\ell^k$.
Note that, for the product $\prod_p T_p$ of cyclic groups of order
$p-1$ (for~$p\ne2$), this statement is not completely trivial: It
follows from the existence, by the well-known theorem of Dirichlet,
of infinitely many primes $p$ that are congruent to $1\mod \ell^k$,
but not to $1\mod \ell^{k+1}$.

Now suppose that $K$ is an arbitrary number field, with ring 
of integers $\Ocal$.
By class field theory, $A_K$ is the quotient of
the idele class group $C_K=(\prod'_{\gothp\le\infty} K_\gothp^*)/K^*$
of~$K$ by the connected component of the identity.
In the case of imaginary quadratic fields $K$, this connected
component is the subgroup $K_\infty^*=\CC^*\subset C_K$ coming
from the unique infinite prime of $K$, and in this case the
Artin isomorphism for the absolute abelian Galois group $A_K$ of $K$
reads
\begin{equation}
\label{eq:A_K}
A_K= \Khat^*/K^* = \Bigl(\mathop{\prod\nolimits^\prime}_\gothp K_\gothp^*\Bigr)/K^*.
\end{equation}
Here $\Khat^* = \prod^\prime_\gothp K_\gothp^*$ is the group
of \emph{finite} ideles of $K$, that is, 
the restricted direct product of the groups $K_\gothp^*$ at the finite
primes $\gothp$ of $K$, taken with respect to the
unit groups~$\Ocal_\gothp^*$ of the local rings of integers.
For the purposes of this paper, which tries to describe $A_K$ as
a profinite abelian group, it is convenient to treat the isomorphism
for $A_K$ in \eqref{eq:A_K} as an identity --  as we have written it down.

The expression \eqref{eq:A_K} is somewhat more involved than the
corresponding identity $A_\Que=\Zhat^*$ for the rational number field,
but we will show in Lemma \ref{lemma:T_K}
that the \emph{inertial part} of $A_K$,
that is, the subgroup $U_K\subset A_K$
generated by all inertia groups $\Ocal_\gothp^*\subset C_K$, admits
a description very similar to \eqref{eq:A_Q}.

Denote by $\Ocalhat=\prod_\gothp \Ocal_\gothp$ the profinite
completion of the ring of integers $\Ocal$ of $K$.
In the case that $K$ is imaginary quadratic,
the inertial part of $A_K$ takes the form
\begin{equation}
\label{eq:inertial}
U_K = \Bigl(\prod_\gothp \Ocal_\gothp^*\Bigr)/\Ocal^* = \Ocalhat^*/\mu_K,
\end{equation}
since the unit group $\Ocal^*$ of $\Ocal$ is then equal to the
group $\mu_K$ of roots of unity in~$K$.
Apart from the quadratic fields of discriminant $-3$ and $-4$, which have
6 and 4 roots of unity, respectively, we always have $\mu_K=\{\pm 1\}$,
and \eqref{eq:inertial} can be viewed as the analogue for~$K$ of the group 
$\widehat\Zee^*=A_\Que$.

In the next section, we determine the structure of the group 
$\Ocalhat^*/\mu_K$. As the approach works for any number field, we will not assume that
$K$ is imaginary quadratic until the very end of that section.

\section{Structure of the inertial part}
\label{S:inertialstructure}

Let $K$ be any number field, and $\Ocalhat=\prod_\gothp \Ocal_\gothp$
the profinite completion of its ring of integers.
Denote by $T_\gothp\subset \Ocal_\gothp^*$ the
subgroup of local roots of unity in $K_\gothp^*$, and put
\begin{equation}
T_K=\prod_\gothp T_\gothp \subset \prod_\gothp \Ocal_\gothp^* = \Ocalhat^*.
\end{equation}
The analogue of \eqref{eq:A_Q} for $K$ is the following.

\begin{lemma}
\label{lemma:Ohatstar}
The closure of the torsion subgroup of $\Ocalhat^*$ is equal to $T_K$,
and $\Ocalhat^*/T_K$ is a free $\Zhat$-module of rank $[K:\Que]$.
Less canonically, we have an isomorphism
$$
\Ocalhat^* \iso T_K \times \Zhat^{[K:\Que]}.
$$
\end{lemma}

\begin{proof}
As the finite torsion subgroup $T_\gothp\subset \Ocal_\gothp^*$ is closed in
$\Ocal_\gothp^*$, the first statement follows from the definition
of the product topology on $\Ocalhat^*=\prod_\gothp \Ocal_\gothp^*$.

Reduction modulo $\gothp$ in the local unit group $\Ocal_\gothp^*$
gives rise to an exact sequence 
$$
1\to 1+\gothp\tto\Ocal_\gothp^*\tto k_\gothp^*\to 1
$$
that can be split by mapping the elements
of the unit group $k_\gothp^*$ of the residue class field 
to their Teichm\"uller representatives in $\Ocal_\gothp^*$.
These form the cyclic group of order 
$\#k_\gothp^*=N\gothp-1$ in $T_\gothp$ consisting of
the elements of order coprime to $p=\characteristic(k_\gothp)$.
The kernel of reduction $1+\gothp$ is by \cite[one-unit theorem, p.~231]{\HA}
a finitely generated $\Zee_p$-module of
free rank $[K_\gothp:\Que_p]$ having a finite torsion group
consisting of roots of unity in $T_\gothp$ of $p$-power order.
Combining these facts, we find that $\Ocal_\gothp^*/T_\gothp$ is free 
over $\Zee_p$ of rank $[K_\gothp:\Que_p]$ or, less canonically, that we
have a local isomorphism
$$
\Ocal_\gothp^* \iso T_\gothp \times \Zee_p^{[K_\gothp:\Que_p]}
$$
for each prime $\gothp$.
Taking the product over all $\gothp$, and using the fact that the sum of the
local degrees at $p$ equals the global degree $[K:\Que]$, 
we obtain the desired global conclusion.
\end{proof}

In order to derive a characterization of $T_K=\prod_\gothp T_\gothp$
for arbitrary number fields~$K$ similar to \eqref{eq:T_Q},
we observe that we have an exact divisibility
$\ell^k\parallel \#T_\gothp$ of the order of $T_\gothp$ by a prime power $\ell^k$
if and only if the local field $K_\gothp$ at $\gothp$ contains
a primitive $\ell^k$-th root of unity,
but \emph{not} a primitive $\ell^{k+1}$-th root of unity.
We may reword this as: The prime $\gothp$ splits completely in the
cyclotomic extension $K\subset K(\zeta_{\ell^k})$, but \emph{not}
in the cyclotomic extension $K\subset K(\zeta_{\ell^{k+1}})$.
If such $\gothp$ exist at all for $\ell^k$, then there are infinitely
many of them, by the Chebotarev density theorem.

Thus, $T_K$ can be written as a product of groups 
$(\Zee/\ell^k\Zee)^\Zee=\hbox{Map}(\Zee, \Zee/\ell^k\Zee)$
that are themselves countable products of cyclic groups of order $\ell^k$.
The prime powers $\ell^k>1$ that occur for $K$ are \emph{all} but
those for which we have an equality
$$
K(\zeta_{\ell^k}) = K(\zeta_{\ell^{k+1}}).
$$
For $K=\Que$ all prime powers $\ell^k$ occur, but for
general $K$, there are finitely many prime powers that may disappear.
This is due to the fact that the infinite cyclotomic extension 
$\Que\subset \Que(\zeta_{\ell^\infty})$ with group $\Zee_\ell^*$ can
partially `collapse' over $K$.

To describe the exceptional prime powers $\ell^k$ that disappear for $K$,
we consider, for $\ell$ an \emph{odd} prime, the number
$$
w(\ell)=w_K(\ell)=\# \mu_{\ell^\infty}\left(K(\zeta_\ell)\right)
$$
of $\ell$-power roots of unity in the field $K(\zeta_\ell)$.
For almost all $\ell$, this number equals $\ell$, and we call
$\ell$ \emph{exceptional} for $K$ if it is divisible by $\ell^2$.
Note that no odd exceptional prime numbers exist for imaginary
quadratic fields $K$.

For the prime $\ell=2$, we consider instead the number
$$
w(2)=w_K(2)=\# \mu_{2^\infty}\left(K(\zeta_4)\right)
$$
of 2-power roots in $K(\zeta_4)=K(i)$.
If $K$ contains $i=\zeta_4$, or if $w(2)$ is divisible by~8, we call 2
\emph{exceptional} for $K$.
Note that the only imaginary quadratic fields $K$ for which 2 is exceptional
are $\Que(i)$ and $\Que(\sqrt{-2})$.

The number $w(K)$ of \emph{exceptional roots of unity} for $K$
is now defined as 
$$
w(K)=\prod_{\ell \text{ exceptional}} w(\ell).
$$
Note that $w(K)$ refers to roots of unity that may or may not be
contained in $K$ itself, and that every prime $\ell$ dividing $w(K)$
occurs with exponent at least 2.
The prime powers $\ell^k>1$ that do \emph{not} occur when 
$T_K$ is written as a direct product of groups $(\Zee/\ell^k\Zee)^\Zee$
are the \emph{strict} divisors of $w(\ell)$ at exceptional primes $\ell$,
with the exceptional prime $\ell=2$ giving rise to a special case.

\begin{lemma}
\label{lemma:T_K}
Let $K$ be a number field, and $w=w(K)$ its number of exceptional roots
of unity.
Then we have a noncanonical isomorphism of profinite groups
$$
T_K=\prod_{\gothp} T_\gothp \iso \prod_{n \geq 1} \Zee /nw\Zee,
$$
except when $2$ is exceptional for $K$ and $i=\zeta_4$ is not contained in $K$.
In this special case, we have
$$
T_K=\prod_{\gothp} T_\gothp \iso \prod_{n \geq 1}
(\Zee/2\Zee\times \Zee / nw\Zee) .
$$
The group $T_K$ is isomorphic to the group $T_\Que$ in
\textup{\eqref{eq:T_Q}} if and only if we have $w=1$.
\end{lemma}

\begin{proof}
If $\ell$ is odd, the tower of field extensions
$$
K(\zeta_\ell) \subset K(\zeta_{\ell^2})\subset \cdots \subset
K(\zeta_{\ell^k}) \subset K(\zeta_{\ell^{k+1}}) \subset \cdots
$$
is a $\Zee_\ell$-extension, and the steps
$K(\zeta_{\ell^k}) \subset K(\zeta_{\ell^{k+1}})$ with $k\ge1$ in this tower
that are equalities are exactly those for which $\ell^{k+1}$ divides $w(\ell)$.

Similarly, the tower of field extensions 
$$
K(\zeta_4) \subset K(\zeta_8)\subset \cdots \subset
K(\zeta_{2^k}) \subset K(\zeta_{2^{k+1}}) \subset \cdots
$$
is a $\Zee_2$-extension in which the steps
$K(\zeta_{2^k}) \subset K(\zeta_{2^{k+1}})$ with $k\ge2$ that are equalities
are exactly those for which $2^{k+1}$ divides $w(2)$.
The extension $K=K(\zeta_2)\subset K(\zeta_4)$ that we have in the remaining
case $k=1$ is an equality if and only if $K$ contains $i=\zeta_4$.

Thus, a prime power $\ell^k>2$ that does not occur when $T_K$ is
written as a product of groups $(\Zee/\ell^k\Zee)^\Zee$ is the same as
a \emph{strict} divisor $\ell^k>2$ of $w(\ell)$ at an exceptional
prime $\ell$.
The special prime power $\ell^k=2$ does not occur if and only if 
$i=\zeta_4$ is in~$K$.
Note that in this case, 2 is by definition exceptional for $K$.

It is clear that replacing the group $\prod_{n \geq 1} \Zee /n\Zee$
from \eqref{eq:T_Q} by $\prod_{n \geq 1} \Zee /nw\Zee$ has the effect
of removing cyclic summands of order $\ell^k$ with $\ell^{k+1}\mid w$,
and this shows that the groups given in the Lemma are indeed isomorphic
to $T_K$.
Only for $w=1$ we obtain the group $T_\Que$ in which all prime powers
$\ell^k$ arise.
\end{proof}

Lemmas \ref{lemma:Ohatstar} and \ref{lemma:T_K} tell us what 
$\Ocalhat^*$ looks like as a $\Zhat$-module.
In particular, it shows that the dependence on $K$ is limited 
to the degree $[K:\Que]$, which is reflected in the rank of the
free $\Zhat$-part of $\Ocalhat^*$, and the nature of the
exceptional roots of unity for $K$.
For the group $\Ocalhat^*/\mu_K$, the same is true,
but the proof requires an extra argument, and the following lemma.

\begin{lemma}
\label{lemma:goodp}
There are infinitely many primes $\gothp$ of $K$ for which we have 
$$
\gcd(\#\mu_K, \#T_\gothp/\#\mu_K) = 1.
$$
\end{lemma}

\begin{proof}
For every prime power $\ell^k>1$ that exactly divides $\#\mu_K$, the
extension $K=K(\zeta_{\ell^k}) \subset K(\zeta_{\ell^{k+1}})$
is a cyclic extension of prime degree $\ell$.
For different prime powers $\ell^k\parallel \#\mu_K$, we get
different extensions, so infinitely many primes~$\gothp$ of~$K$
are inert in all of them.
For such $\gothp$, we have $\gcd(\#\mu_K, \#T_\gothp/\#\mu_K) = 1$.
\end{proof}

\begin{lemma}
We have a noncanonical isomorphism
$T_K/\mu_K \iso T_K$.
\end{lemma}
\begin{proof}
Pick a prime $\gothp_0$ of $K$ that satisfies
the conditions of Lemma \ref{lemma:goodp}.
Then $\mu_K$ embeds as a direct summand in $T_{\gothp_0}$,
and we can write $T_{\gothp_0} \iso \mu_K \times T_{\gothp_0}/\mu_K$
as a product of two cyclic groups of coprime order.
It follows that the natural exact sequence
$$
1\to \prod_{\gothp\ne\gothp_0} T_\gothp
\tto  T_K/\mu_K \tto T_{\gothp_0}/\mu_K \to 1
$$
can be split using the composed map
$T_{\gothp_0}/\mu_K \to T_{\gothp_0} \to T_K \to T_K/\mu_K$.
This makes $T_K/\mu_K$ isomorphic to the product of
$\prod_{\gothp\ne\gothp_0} T_\gothp$
and a cyclic group for which the order is a product of prime powers
that already `occur' infinitely often in $T_K$.
Thus $T_K/\mu_K$ is isomorphic to a product of exactly the same 
groups $(\Zee/\ell^k\Zee)^\Zee$ that occur in $T_K$, and therefore
isomorphic to $T_K$ itself.
\end{proof}

For imaginary quadratic $K$, where $\Ocalhat^*/\mu_K$
constitutes the inertial part $U_K$ of~$A_K$
from \eqref{eq:inertial}, we
summarize the results of this section in the following way.

\begin{theorem}
\label{thm:U_K}
Let $K$ be an imaginary quadratic field.
Then the subgroup $T_K/\mu_K$ of $U_K$ is a direct summand of~$U_K$.
For $K\ne \Que(i), \Que(\sqrt{-2})$,
we have isomorphisms
$$
U_K= \Ocalhat^*/\mu_K \iso
\Zhat^2 \times (T_K/\mu_K) 
\iso
\Zhat^2 \times \prod_{n = 1}^{\infty} \Zee / n\Zee 
$$ 
of profinite groups.
\end{theorem}
For $K$ equal to $\Que(i)$ or $\Que(\sqrt{-2})$, 
the prime 2 is exceptional for $K$,
and the groups $T_K/\mu_K\iso T_K$ are different as they
do not have cyclic summands of order~2 and~4, respectively.

\section{Extensions of Galois groups}
\label{S:extensions}
In the previous section, all results could easily be stated and
proved for arbitrary number fields.
From now on, $K$ will denote an imaginary quadratic field.
In order to describe the full group $A_K$
from \eqref{eq:A_K}, we consider the exact sequence
\begin{equation}
\label{eq:extension}
1\to U_K=\Ohat^*/\mu_K
\tto A_K = \Khat^*/K^*
\mapright{\psi} \Cl_K \to 1
\end{equation}
that describes the class group $\Cl_K$ of $K$ in idelic terms. 
Here $\psi$ maps the class of the finite idele $(x_\gothp)_\gothp\in\Khat^*$
to the class of its associated ideal
$\prod_\gothp \gothp^{e_\gothp}$, with $e_\gothp= \ord_\gothp x_\gothp$.

The sequence \eqref{eq:extension} shows that $U_K$ is an open  subgroup of $A_K$
of index equal to the class number $h_K$ of $K$.
In view of Theorem \ref{thm:U_K}, this immediately yields Onabe's discovery
that different $K$ can have the same absolute abelian Galois group.

\begin{theorem}
\label{thm:h=1}
An imaginary quadratic number field $K\ne\Que(i), \Que(\sqrt{-2})$
of class number $1$
has absolute abelian Galois group isomorphic to
$$
G=\Zhat^2 \times \prod_{n\ge1} \Zee/n\Zee.
$$
\end{theorem}

In Onabe's paper \cite[\S5]{\ONa}, the group $G$, which is not explicitly
given but characterized by its infinitely many Ulm invariants,
is referred to as `of type A'.
We will refer to $G$ as the \emph{minimal} Galois group, as every 
absolute abelian Galois group of an imaginary quadratic field 
$K\ne\Que(i), \Que(\sqrt{-2})$ contains a subgroup isomorphic to~$G$.
We will show that there are actually \emph{many} more $K$ having this
absolute abelian Galois group than the seven fields $K$ of class number 1
to which the preceding theorem applies.

Let us now take for $K$ any imaginary quadratic field of
class number $h_K>1$.
Then Theorem \ref{thm:U_K} and the sequence \eqref{eq:extension} show
that $A_K$ is an abelian group extension of $\Cl_K$ by the minimal
Galois group $G$ from Theorem \ref{thm:h=1}.
If the extension \eqref{eq:extension} were split, we would find
that $A_K$ is isomorphic to $G\times\Cl_K\iso G$.
However, it turns out that splitting at this level \emph{never}
occurs for nontrivial $\Cl_K$, in the following strong sense.

\begin{theorem}
\label{thm:nonsplit}
For every imaginary quadratic field $K$ of class number $h_K>1$,
the sequence \textup{\eqref{eq:extension}} is totally nonsplit\textup{;}
that is, there is no nontrivial subgroup $C\subset \Cl_K$ for which
the associated subextension $1\to U_K \to \psi^{-1}[C]\to C\to 1$ is split.
\end{theorem}

\begin{proof}
Suppose there is a non-trivial subgroup $C\subset \Cl_K$ over which
the extension~\eqref{eq:extension} splits, and pick
$[\gotha]\in C$ of prime order $p$.
Then there exists an element
$$
\left((x_\gothp)_\gothp\mod K^*\right)\in \psi^{-1}([\gotha])\subset A_K
=\Khat^*/K^*
$$
of order $p$.
In other words, there exists $\alpha\in K^*$ such that we have
$x_\gothp^p=\alpha\in K_\gothp^*$ for all $\gothp$,
and such that $\alpha$ generates the ideal $\gotha^p$.
But this implies by \cite[Chapter~IX, Theorem~1]{\AT} that
$\alpha$ is a $p$-th power in $K^*$, and hence that $\gotha$
is a principal ideal.
Contradiction.
\end{proof}

At first sight, Theorem \ref{thm:nonsplit}
seems to indicate that in the case $h_K>1$,
the group~$A_K$ will \emph{not} be isomorphic to 
the minimal Galois group $G\iso U_K$.
However, finite abelian groups requiring no more than $k$ generators
do allow extensions by free $\Zhat$-modules of finite rank $k$
that are again free of rank $k$,
just like they do with free $\Zee$-modules in the
classical setting of finitely generated abelian groups.
The standard example for $k=1$ is the extension
$$
1\to\Zhat\mapright{\times p\ }\Zhat\tto\Zee/p\Zee\to 1
$$
for an integer $p\ne0$, prime or not.
Applying to this the functor $\Hom(-, M)$ for a multiplicatively written
$\Zhat$-module $M$, we obtain an isomorphism
\begin{equation}
\label{eq:ext-z/nz}
M/M^p \isar \Ext(\Zee/p\Zee, M)
\end{equation}
by the Hom-Ext-sequence from homological algebra \cite{\MA}.
We will use it in Section~\ref{S:algorithm}.

\begin{lemma}
\label{lem:totnonsplit}
Let $B$ be a finite abelian group, $F$ a free $\Zhat$-module 
of finite rank $k$, and
$$
1\to F \tto E \tto B\to 1
$$
an exact sequence of $\Zhat$-modules.
Then $E$ is free of rank $k$ if and only if this sequence is totally nonsplit.
\end{lemma}

\begin{proof}
One may reduce the statement
to the familiar case of modules over principal ideal
domains by writing $\Zhat=\prod_p \Zee_p$, and consider the individual
$p$-parts of the sequence.
As a matter of convention, note
that in the degenerate case where $B$ is the trivial group,
there are no nontrivial subgroups $C\subset B$ over which the 
sequence splits, making the sequence by definition totally nonsplit.
\end{proof}

In order to apply the preceding lemma,
we replace the extension \eqref{eq:extension}
by the pushout under the quotient map
$U_K=\Ocalhat^*/\mu_K \to U_K/T_K = \Ocalhat^*/T_K$
from $U_K$ to its maximal $\Zhat$-free quotient.
This yields the exact sequence of $\Zhat$-modules
\begin{equation}
\label{eq:pushout}
1\to \Ohat^*/T_K
\tto \Khat^*/(K^*\cdot T_K)
\tto \Cl_K \to 1
\end{equation}
in which $\Cl_K$ is finite and 
$\Ohat^*/T_K$ is free of rank 2 over $\Zhat$ by Lemma \ref{lemma:Ohatstar}.

\begin{theorem}
\label{thm:minG}
Let $K$ be an imaginary quadratic field of class number $h_K>1$, 
and suppose the sequence \textup{\eqref{eq:pushout}} is totally
nonsplit.  Then the absolute abelian Galois group of $K$ is
the minimal group~$G$ occurring in Theorem $\ref{thm:h=1}$.
\end{theorem}

\begin{proof}
If the extension \textup{\eqref{eq:pushout}} is totally nonsplit, then
$\Khat^*/(K^*\cdot T_K)$ is free of rank 2 over $\Zhat$
by Lemma \ref{lem:totnonsplit}.
In this case the exact sequence of $\Zhat$-modules
$$
1\to T_K/\mu_K \tto A_K = \Khat^*/K^* \tto \Khat^*/(K^*\cdot T_K) \to 1
$$
is split, and $A_K$ is isomorphic to $U_K= G = \Zhat^2 \times (T_K/\mu_K)$.
\end{proof}

\begin{remark*}
We will use Theorem \ref{thm:minG} in this paper
to find many imaginary quadratic fields
$K$ having the same minimal absolute abelian Galois group $G$.
It is however interesting to note that this is the only way in
which this can be done, as Theorem~\ref{thm:minG} actually
admits a converse: If the absolute abelian Galois group of
an imaginary quadratic field $K$ of class number $h_K>1$ is
the minimal group $G$, then 
the sequence \eqref{eq:pushout} is totally nonsplit.
The proof, which we do not include in this paper, will be given in
the forthcoming doctoral thesis of the first author.
\end{remark*}

It is instructive to see what all the preceding extensions of
Galois groups amount to in terms of field extensions.
The diagram of fields 
in Figure~\ref{Fig:fields}
lists all subfields of the extension
$K\subset K^\ab$ corresponding to the various subgroups we considered
in analyzing the structure of $A_K=\Gal(K^\ab/K)$.

We denote by $H$ the Hilbert class field of $K$.
This is the maximal totally unramified abelian extension of $K$, and it 
is finite over $K$ with group $\Cl_K$.
The inertial part of $A_K$ is the Galois group
$U_K=\Gal(K^\ab/H)$, which is isomorphic to $G$ for all 
imaginary quadratic fields $K\ne\Que(i), \Que(\sqrt{-2})$.
The fundamental sequence \eqref{eq:extension} corresponds to the tower
of fields 
$$
K\subset H\subset K^\ab.
$$
By Theorem \ref{thm:U_K}, the invariant field $L$ of 
the closure $T_K/\mu_K$ of the torsion subgroup of $U_K$
is an extension of $H$ with group $\Zhat^2$.
The tower of field extensions 
$$
K\subset H\subset L
$$
corresponds to
the exact sequence of Galois groups \eqref{eq:pushout}.

We define $L_0$ as the `maximal $\Zhat$-extension' of $K$, that is, as the
compositum of the $\Zee_p$-extensions of $K$ for \emph{all} primes $p$.
As is well-known, an imaginary quadratic field admits two independent
$\Zee_p$-extensions for each prime $p$, so
$F=\Gal(L_0/K)$ is a free $\Zhat$-module of rank 2,
and $L_0$ is the invariant field under the closure $T_0$
of the torsion subgroup of $A_K$.
The image of the restriction map $T_0 \to \Cl_K$ is 
the maximal subgroup of $\Cl_K$ over which \eqref{eq:pushout} splits.
The invariant subfield of $H$ corresponding to it is
the intersection $L_0\cap H$.
The totally nonsplit case occurs when $H$ is contained in $L_0$,
leading to $L_0\cap H=H$ and $L_0=L$.
In this case $\Gal(L/K)=\Gal(L_0/K)$ is itself a free $\Zhat$-module of rank 2, and
$A_K$ is an extension of $\Zhat^2$ by $T_K/\mu_K$ that is 
isomorphic to $G$.

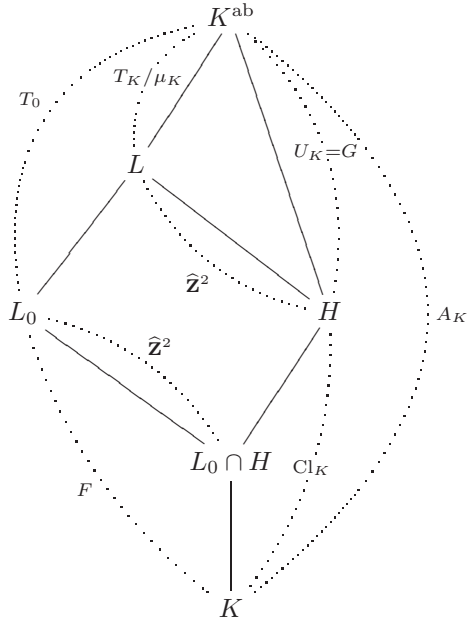
\begin{figure}[ht]
\begin{center}
$$
\xymatrixrowsep{0.25in}
\xymatrixcolsep{0.15in}
\xymatrix{
 & & & & &  & K^{\textrm{ab}}\ar@{-}[ddl] \ar@{-}[ddddr] \ar@/^1.5pc/@{.}[ddddr]|-{U_K=G} \ar@/^6.2pc/@{.}[dddddddd]^{A_K} \ar@/_1.2pc/@{.}[ddl]|-{T_K/\mu_K} \ar@/_3pc/@{.}[ddddlll]_{T_{0}}  &  & & \\
 & & & & &  & & & & &  & \\
 & & & & & L\ar@{-}[ddll] \ar@{-}[ddrr] \ar@/_1pc/@{.}[ddrr]_{\Zhat^{2}} &  & & &  \\
 & & & & & & & & & & \\
 & & & L_{0}\ar@{-}[ddrrr] \ar@/^1pc/@{.}[ddrrr]^{\Zhat^{2}} \ar@/_1pc/@{.}[ddddrrr]_{F} & &   & & H\ar@{-}[ddl] \ar@/^1pc/@{.}[ddddl]|-{\textrm{Cl}_{K}} & \\
 & & & & &  & & & & & \\
 & & & & & & L_{0}\cap H  \ar@{-}[dd] & & & \\
 & & & & & & & & & & \\
 & & & & & & K & & & \\
}$$
\caption{The structure of $A_K = \Gal(K^{\textrm{ab}}/K)$.}
\label{Fig:fields}
\end{center}
\end{figure}

\section{Finding minimal Galois groups}
\label{S:algorithm}
In order to use Theorem \ref{thm:minG}
and find imaginary quadratic $K$ for which the 
absolute abelian Galois group $A_K$ is the minimal group $G$
from Theorem \ref{thm:h=1},
we need an algorithm that can effectively
determine, on input $K$, whether the sequence of $\Zhat$-modules
$$
1\to \Ohat^*/T_K
\tto \Khat^*/(K^*\cdot T_K)
\tto \Cl_K \to 1
\leqno{\eqref{eq:pushout}}
$$
from Section \ref{S:extensions} is totally nonsplit.
This means that for every ideal class $[\gotha]\in\Cl_K$
of prime order, the 
subextension $E[\gotha]$ of \eqref{eq:pushout} lying over 
the subgroup $\langle[\gotha]\rangle\subset\Cl_K$ is nonsplit.

Any profinite abelian group $M$ is a module
over $\Zhat=\prod_p \Zee_p$, and can be written accordingly as a 
product $M=\prod_p M_p$ of $p$-primary parts, where
$M_p=M\otimes_\Zhat \Zee_p$ is a pro-$p$-group and $\Zee_p$-module.
In the same way, an exact sequence of $\Zhat$-modules is
a `product' of exact sequences for their $p$-primary parts, and
splitting over a group of prime order $p$ only involves 
$p$-primary parts for that $p$.

For the free $\Zhat$-module $M=\Ohat^*/T_K$ in \eqref{eq:pushout},
we write $T_p$ for the torsion subgroup of
$\Ocal_p^*=(\Ocal\otimes_\Zee \Zee_p)^*=\prod_{\gothp\mid p} \Ocal_\gothp^*$.
Then the $p$-primary part of $M$
is the pro-$p$-group
\begin{equation}
\label{eq:O_p/T_p}
M_p=\Ocal_p^*/T_p = \prod_{\gothp\mid p} (\Ocal_\gothp^*/T_\gothp) \iso \Zee_p^2.
\end{equation}
In order to verify the hypothesis of Theorem \ref{thm:minG},
we need to check that the extension $E[\gotha]$ has nontrivial
class in $\Ext(\langle[\gotha]\rangle, M)$ for all
$[\gotha]\in\Cl_K$ of prime order~$p$.
We can do this by verifying in each case that the element of $M/M^p=M_p/M_p^p$
corresponding to it under the isomorphism~\eqref{eq:ext-z/nz}
is nontrivial.
This yields the following theorem.

\begin{theorem}
\label{thm:criterion}
Let $K$ be an imaginary quadratic field, and define for each prime 
number $p$ dividing $h_K$ the homomorphism
$$
\phi_p: \Cl_K[p]\tto \Ocal_p^*/T_p (\Ocal_p^*)^p
$$
that sends the class of a $p$-torsion ideal $\gotha$ coprime to $p$
to the class of a generator of the ideal $\gotha^p$.
Then \textup{\eqref{eq:pushout}} is totally nonsplit if and only if
all maps $\phi_p$ are injective.
\end{theorem}

\begin{proof}
Under the isomorphism \eqref{eq:ext-z/nz}, the class of the
extension
$$
1\to M\tto E\mapright{f}\Zee/p\Zee\to 1
$$
in $\Ext(\Zee/p\Zee, M)$ corresponds 
by \cite[Chapter III, Proposition~1.1]{\MA} to the residue
class of the element $(f^{-1}(1\mod p\Zee))^p\in M/M^p$.
In the case of $E[\gotha]$, we apply this to
$M=\Ocalhat^*/T_K$, and choose the identification
$\Zee/p\Zee=\langle[\gotha]\rangle$ under which $1\mod p\Zee$ is
the \emph{inverse} of $[\gotha]$.
Then $f^{-1}(1\mod p\Zee)$ is the residue class in
$\Khat^*/(K^*\cdot T_K)$ of any finite idele $x\in\Khat^*$
that is mapped to ideal class of $\gotha^{-1}$
under the map $\psi$ from \eqref{eq:extension}.

We pick $\gotha$ in its ideal class coprime to $p$,
and take for $x=(x_\gothp)_\gothp$ an idele that locally
generates $\gotha^{-1}$ at all $p$.
If $\alpha\in K^*$ generates $\gotha^p$,
then $x^p\alpha$ is an idele in $\Ocalhat^*$ that lies 
in the same class modulo $K^*$ as $x^p$, and its image 
$$
(f^{-1}(1\mod p\Zee))^p=x^p=x^p\alpha
\in M/M^p=M_p/M_p^p=\Ocal_p^*/T_p (\Ocal_p^*)^p
$$
corresponds to the class of $E[\gotha]$ in
$\Ext(\langle[\gotha]\rangle, \Ocal^*/T_K)$.
As the idele $x=(x_\gothp)_\gothp$ has components $x_\gothp\in\Ocal_\gothp^*$
at $\gothp\mid p$ by the choice of $\gotha$, we see that this image
in $M_p/M_p^p=\Ocal_p^*/T_p (\Ocal_p^*)^p$
is the element $\phi_p([\gotha])$ we defined.
The map $\phi_p$ is clearly a homomorphism, and we want it to assume nontrivial
values on the elements of order $p$ in $\Cl_K[p]$, for each prime
$p$ dividing $h_K$.
The result follows.
\end{proof}

\begin{remark*}
In Theorem~\ref{thm:criterion}, it is not really necessary to restrict to
representing ideals $\gotha$ that are coprime to $p$.
One may take $K_p^*/T_p(K_p^*)^p$ as the target space of $\phi_p$
to accommodate all $\gotha$, with $K_p=K\otimes_\Zee \Zee_p$,
and observe that the image of $\phi_p$ is in the subgroup
$\Ocal_p^*/T_p (\Ocal_p^*)^p$, as the valuations of $\gotha^p$
at the primes over $p$ are divisible by~$p$.
\end{remark*}

\begin{remark*}
It is possible to prove Theorem \ref{thm:criterion} without explicit
reference to homological algebra.
What the proof shows is that, in order to lift an ideal class of
arbitrary order $n$ under \eqref{eq:pushout}, it is necessary and sufficient
that its $n$-th power is generated by an element $\alpha$ that is
locally everywhere a $n$-th power \emph{up to multiplication by
local roots of unity.}
This extra leeway in comparison with the situation in
Theorem \ref{thm:nonsplit} makes it into an interesting 
splitting problem for the group extensions involved, as
this condition on $\alpha$ may or may not be satisfied.
Note that at primes outside $n$,
the divisibility of the valuation of $\alpha$ by $n$ automatically
implies the local condition.

In Onabe's paper, which assumes throughout that $\Cl_K$ itself
is a cyclic group of prime order, the same criterion is obtained
from an analysis of the Ulm invariants occurring in Kubota's
setup \cite{\KU}.
\end{remark*}

Our Theorem \ref{thm:criterion} itself does not assume any
restriction on $\Cl_K$, but its use in finding $K$ with 
minimal absolute Galois group $G$ does imply certain
restrictions on the structure of $\Cl_K$.
The most obvious implication of the injectivity of the map
$\phi_p$ in the theorem is a bound on the
$p$-rank of $\Cl_K$, which is defined as the dimension
of the group $\Cl_K/\Cl_K^p$ as an $\FF_p$-vector space.

\begin{corollary}
\label{cor:rank3}
If\/ $\Cl_K$ has $p$-rank at least $3$ for some $p$, then
the sequence \textup{\eqref{eq:pushout}} splits over a subgroup of
$\Cl_K$ of order $p$.
\end{corollary}

\begin{proof}
It follows from the isomorphism in 
\eqref{eq:O_p/T_p} that the image of $\phi_p$ lies in
a group that is isomorphic to $(\Zee/p\Zee)^2$.
If $\Cl_K$ has $p$-rank at least $3$, then $\phi_p$ will not
be injective.
Now apply Theorem \ref{thm:criterion}.
\end{proof}

As numerical computations in uncountable \hbox{$\Zhat$-modules} such 
as $\Khat^*/(K^*\cdot T_K)$ can only be performed
with finite precision, it is not immediately obvious that the 
splitting type of an idelic extension as \eqref{eq:pushout}
can be found by a finite computation.
The maps $\phi_p$ in Theorem \ref{thm:criterion} however
are linear maps between finite-dimensional $\FF_p$-vector spaces that
lend themselves very well to explicit computations.
One just needs some standard algebraic number theory to compute
these spaces explicitly.
A high-level description of an algorithm that 
determines whether the extension \eqref{eq:pushout} is 
totally nonsplit is then easily written down.

\begin{algorithm}
\label{Alg}
\begin{alglist}
\algin An imaginary quadratic number field $K$.
\algout \emph{No} if the extension \textup{\eqref{eq:pushout}} for $K$
is not totally nonsplit, \emph{yes} otherwise.
\item Compute the class group $\Cl_K$ of $K$. 
           If $\Cl_K$ has $p$-rank at least $3$ for some~$p$,
           output \emph{no} and stop.
\item For each prime $p$ dividing $h_K$,
compute $n\in\{1,2\}$ $\Ocal$-ideals 
coprime to $p$ such that their classes in $\Cl_K$ generate $\Cl_K[p]$,
and generators $x_1$ up to $x_n$ for their $p$-th powers. 

Check whether $x_1$ is trivial in 
$\Ocal_{(p)}^*/T_K(\Ocal_{(p)}^*)^p$.
If it is, output \emph{no} and stop.
If $n=2$, check whether $x_2$ is trivial in 
$\Ocal_p^*/T_K\cdot\langle x_1\rangle\cdot(\Ocal_p^*)^p$.
If it is, output \emph{no} and stop.
\item If all primes $p\mid h_K$ are dealt with without
stopping, output \emph{yes} and stop.
\end{alglist}
\end{algorithm}

Step 1 is a standard task in computational algebraic
number theory.
For imaginary quadratic fields, it is often implemented
in terms of binary quadratic forms, and particularly easy.
From an explicit presentation of the group, 
it is also standard to find the global
elements $x_1$ and, if needed, $x_2$.
The rest of Step 2 takes place in a \emph{finite} group,
and this means that we only compute in the 
rings $\Ocal_p$ up to small precision.
For instance, computations in $\Zee_p^*/T_p(\Zee_p^*)^p$
amount to computations modulo~$p^2$ for odd~$p$, and modulo
$p^3$ for $p=2$.

\section{Splitting behavior at 2}
\label{S:splittingat2}
The splitting behavior of the sequence \eqref{eq:pushout}
depends strongly on the structure of the $p$-primary parts of
$\Cl_K$ at the primes $p\mid h_K$.
In view of Theorem \ref{thm:criterion} and Corollary~\ref{cor:rank3},
fields with cyclic class groups and few small primes dividing $h_K$
appear to be more likely to have minimal Galois group $G$.
In Section \ref{S:computations}, we will provide numerical data
to examine the average splitting behavior.

For odd primes $p$, class groups of $p$-rank at least 3
arising in Corollary \ref{cor:rank3} are very rare, at
least numerically and according to the Cohen-Lenstra heuristics.
At the prime 2, the situation is a bit different, as the 
2-torsion subgroup of $\Cl_K$ admits a classical explicit
description going back to Gauss.
Roughly speaking, his theorem on ambiguous ideal classes states
that $\Cl_K[2]$ is an $\FF_2$-vector space generated by the classes
of the primes $\gothp$ of $K$ lying over the rational primes that
ramify in $\Que\subset K$, subject to a single relation coming from
the principal ideal $(\sqrt {D_K})$.
Thus, the \makebox{2-\hfill rank} of $\Cl_K$ for a discriminant with $t$ distinct prime
divisors equals $t-1$.
In view of Corollary \ref{cor:rank3}, our method to construct $K$
with absolute abelian Galois group $G$ does not apply if the discriminant
$D_K$ of $K$ has more than~3 distinct prime divisors.

If $-D_K$ is a prime number, then $h_K$ is odd, and there is
nothing to check at the prime~$2$.

For $D_K$ with two distinct prime divisors, the 2-rank of $\Cl_K$ equals 1,
and we can replace the computation at $p=2$ in Algorithm \ref{Alg} by
something that is much simpler.

\begin{theorem}
\label{thm:2rank1}
Let $K$ be an imaginary quadratic field with even class number, and suppose
that its $2$-class group is cyclic.
Then the sequence \textup{\eqref{eq:pushout}} is nonsplit over $\Cl_K[2]$
if and only if the discriminant $D_K$ of $K$ is of one of the
following types\textup{:}
\begin{enumerate}
\item $D_K=-pq$ for primes $p\congr -q\congr 5\mod 8$\textup{;}
\item $D_K=-4p$ for a prime $p\congr 5\mod8$\textup{;}
\item $D_K=-8p$ for a prime $p\congr \pm3\mod8$.
\end{enumerate}
\end{theorem}

\begin{proof}
If $K$ has a nontrivial cyclic 2-class group, then $D_K\congr 0, 1\mod 4$
is divisible by exactly two different primes.

If $D_K$ is odd, we have $D_K=-pq$ for primes
$p\congr 1\mod 4$ and $q\congr 3\mod 4$, and the ramified
primes $\gothp$ and $\gothq$ of $K$ are in the unique ideal class
of order 2 in $\Cl_K$.
Their squares are ideals generated by the integers $p$ and $-q$
that become squares in the genus field $F=\Que(\sqrt p, \sqrt{-q})$
of $K$, which is a quadratic extension of $K$ with group $C_2\times C_2$
over $\Que$ that is locally unramified at 2.

If we have $D_K\congr 5\mod 8$, then 2 is inert in $\Que\subset K$, and
2 splits in $K\subset F$.
This means that $K$ and $F$ have isomorphic completions at
their primes over 2, and that $p$ and $-q$ are local squares at 2.
In this case $\phi_2$ is the trivial map
in Theorem~\ref{thm:criterion}, and is not injective.

If we have $D\congr 1\mod 8$ then 2 splits in $\Que\subset K$.
In the case  $p\congr -q\congr 1\mod 8$ the integers
$p$ and $-q$ are squares in $\Zee_2^*$, and $\phi_2$ is again the
trivial map.
In the other case $p\congr -q\congr 5\mod 8$, the generators
$p$ and $-q$ are nonsquares in $\Zee_2^*$, also up to multiplication
by elements in $T_2=\{\pm1\}$.
In this case $\phi_2$ is injective.

If $D_K$ is even, we either have $D_K=-4p$ for a prime 
$p\congr 1\mod 4$ or $D_K=-8p$ for an odd prime $p$.
In the case $D_K=-4p$ the ramified prime over 2 is in the ideal class
of order 2.
For $p\congr 1\mod 8$, the local field $\Que_2(\sqrt{-p})=\Que_2(i)$
contains a square root of $2i$, and $\phi_2$ is not injective.
For $p\congr 5\mod 8$, the local field $\Que_2(\sqrt{-p})=\Que_2(\sqrt 3)$
does not contain a square root of $\pm 2$, and $\phi_2$ is injective.
In the case $D_K=-8p$ the ramified primes over both 2 and $p$ are in the
ideal class of order 2.
For $p\congr\pm1\mod 8$ the generator $\pm p$ is a local square at 2.
For $p\congr\pm3\mod 8$ it is not.
\end{proof}

In the case where the 2-rank of $\Cl_K$ exceeds 1, the situation
is even simpler.

\begin{theorem}
\label{thm:2rank2}
Let $K$ be an imaginary quadratic field for which the $2$-class group
is noncyclic.
The the map $\phi_2$ in Theorem $\ref{thm:criterion}$
is not injective.
\end{theorem}

\begin{proof}
As every 2-torsion element in $\Cl_K$ is the class of a ramified prime
$\gothp$, its square can be generated by a rational prime number.
This implies that the image of $\phi_2$ is contained in the 
cyclic subgroup 
$$
\Zee_2^*/\{\pm1\}(\Zee_2^*)^2 \subset 
\Ocalhat^*/T_2(\Ocalhat^*)^2
$$
of order 2.
Thus $\phi_2$ is not injective if $\Cl_K$ has noncyclic 2-part.
\end{proof}

In view of Theorem \ref{thm:minG} and the remark following it,
imaginary quadratic fields $K$ for which $A_K$ is the minimal 
Galois group from Theorem~\ref{thm:h=1} can only be found among 
those $K$ for which $-D_K$ is prime, or in the infinite families 
from Theorem~{thm:2rank1}.  In the next section, we will find
many of such $K$.

\section{Computational results}
\label{S:computations}
In Onabe's paper \cite{\ONa}, only cyclic class groups $\Cl_K$ of
prime order $p\le 7$ are considered.
In this case there are just 2 types of splitting behavior
for the extension~\eqref{eq:pushout}, and Onabe provides a list
of the first few $K$ with $h_K=p\le 7$, together with the type of
splitting they represent.
For $h_K=2$ the list is in accordance with Theorem~\ref{thm:2rank1}.
In the cases $h_k=3$ and $h_K=5$ there are only 2 split examples against
10 and 7 nonsplit examples, and for $h_K=7$ no nonsplit examples are
found.
This suggests that $\phi_p$ is rather likely to be injective for
increasing values of $h_K=p$.

This belief is confirmed if we extend Onabe's list by including
\emph{all} imaginary quadratic $K$ of odd prime class number
$h_K=p<100$.
By the work of Watkins~\cite{\WA}, we now know, much more precisely
than Onabe did, what the exact list of fields with given small
class number looks like.
The extended list, with the 65 out of 2356 cases in which the extension
\eqref{eq:pushout}
splits mentioned explicitly, is given in Table~\ref{Table:types}.

As the nonsplit types give rise to fields $K$ having
the minimal group $G$ as its absolute Galois group,
one is inevitably led to the following conjecture.
\begin{conjecture}
\label{conj:infmany}
There are infinitely many imaginary quadratic fields $K$
for which the absolute abelian Galois group is isomorphic to
$$
G=\Zhat^2 \times \prod_{n\ge1} \Zee/n\Zee.
$$
\end{conjecture}

The numerical evidence may be strong, but
we do not even have a theorem that there are infinitely many
prime numbers that occur as the class number of an imaginary quadratic field.
And even if we had, we have no theorem telling us what the distribution
between split and nonsplit will be.

\begin{table}[ht]
\begin{center}
\begin{tabular}{rrrrcl}
\toprule
$p$ &\multicolumn{2}{c}{$\#\{K : h_K = p\}$}& \multicolumn{2}{c}{\#Nonsplit} & \qquad $-D_K$ for split $K$ \\
\midrule
 2 & \quad\qquad 18 && \qquad\ 8 && $     35,     51,     91,   115,   123,  187, 235, 267, 403, 427$\\
 3 &             16 &&        13 && $    107,    331,    643                                        $\\
 5 &             25 &&        19 && $    347,    443,    739,  1051,  1123, 1723                    $\\
 7 &             31 &&        27 && $    859,   1163,   2707,  5107                                 $\\
11 &             41 &&        36 && $   9403,   5179,   2027, 10987, 13267                          $\\ 
13 &             37 &&        34 && $   1667,   2963,  11923                                        $\\
17 &             45 &&        41 && $    383,   8539,  16699, 25243                                 $\\
19 &             47 &&        43 && $   4327,  17299,  17539, 17683                                 $\\ 
23 &             68 &&        65 && $   2411,   9587,  21163                                        $\\
29 &             83 &&        80 && $  47563,  74827, 110947                                        $\\
31 &             73 &&        70 && $   9203,  12923,  46867                                        $\\ 
37 &             85 &&        83 && $  20011,  28283                                                $\\
41 &            109 &&       106 && $  14887,  21487,  96763                                        $\\
43 &            106 &&       105 && $  42683                                                        $\\ 
47 &            107 &&       107 && ---                                                              \\
53 &            114 &&       114 && ---                                                              \\
59 &            128 &&       126 && $ 125731, 166363                                                $\\ 
61 &            132 &&       131 && $ 101483                                                        $\\
67 &            120 &&       119 && $ 652723                                                        $\\
71 &            150 &&       150 && ---                                                              \\
73 &            119 &&       117 && $ 358747, 597403                                                $\\
79 &            175 &&       174 && $  64303                                                        $\\
83 &            150 &&       150 && ---                                                              \\
89 &            192 &&       189 && $  48779, 165587, 348883                                        $\\
97 &            185 &&       184 && $ 130051                                                        $\\
\bottomrule
\end{tabular}
\bigskip
\end{center}
\caption{Splitting types for fields $K$ with $h_K=p < 100$.
The second column gives the number of imaginary quadratic fields with class
number~$p$; the third column gives the number of such fields for which
the sequence~\eqref{eq:pushout} does not split; and the fourth column
gives $-D_K$ for the fields $K$ for which \eqref{eq:pushout} splits.}
\label{Table:types}
\end{table}

From Table~\ref{Table:types}, one easily gets the impression that among all $K$
with $h_K=p$, the fraction for which the sequence \eqref{eq:pushout} splits 
is about $1/p$.
In particular, assuming infinitely many imaginary quadratic fields to have 
prime class number, we would expect 100\% of these fields to have the minimal
absolute abelian Galois group $G$.

If we fix the class number $h_K=p$, the list of $K$ will be finite,
making it impossible to study the average distribution of the
splitting behavior over $\Cl_K[p]$.
For this reason, we computed the average splitting behavior over
$\Cl_K[p]$ for the set $S_p$ of imaginary quadratic fields $K$ for
which the class number has a \emph{single} factor~$p$.

\begin{table}[ht]
\begin{center}
\begin{tabular}{rrrccc}
\toprule
$p$ &\qquad$N_p$&& $B_p$ && $p\cdot f_p$ \\
\midrule
 3 &  300 && $10^7$ && 0.960 \\ 
 5 &  500 && $10^7$ && 0.930 \\ 
 7 &  700 && $10^7$ && 0.960 \\ 
11 & 1100 && $10^7$ && 0.990 \\ 
13 & 1300 && $10^7$ && 1.070 \\ 
17 & 1700 && $10^7$ && 0.920 \\ 
19 & 1900 && $10^7$ && 1.000 \\ 
23 & 2300 && $10^7$ && 1.030 \\ 
29 & 2900 && $10^6$ && 1.000 \\ 
31 & 3100 && $10^6$ && 0.970 \\ 
37 & 3700 && $10^6$ && 0.930 \\ 
41 & 4100 && $10^6$ && 1.060 \\ 
43 & 2150 && $10^6$ && 1.080 \\ 
47 &  470 && $10^7$ && 0.900 \\ 
53 &  530 && $10^5$ && 1.000 \\ 
59 &  590 && $10^6$ && 0.900 \\ 
61 & 1830 && $10^5$ && 0.933 \\ 
67 &  670 && $10^6$ && 0.900 \\ 
71 & 1000 && $10^5$ && 1.136 \\ 
73 & 3650 && $10^5$ && 0.900 \\ 
79 & 1399 && $10^7$ && 1.130 \\ 
83 & 1660 && $10^5$ && 1.000 \\ 
89 &  890 && $10^5$ && 1.100 \\ 
97 &  970 && $10^8$ && 1.100 \\ 
\bottomrule
\end{tabular}
\bigskip
\end{center}
\caption{Splitting fractions at $p$ for $h_K$ divisible by $p<100$.
For the given values of $p$, $N_p$, and $B_p$, we consider
the first $N_p$ imaginary quadratic fields $K$ with 
$|D_K|>B_p$ and whose class numbers are divisible by a single factor of $p$.
The fourth column gives the value of $p\cdot f_p$, where $f_p$ is 
the fraction of these fields for which the sequence \eqref{eq:pushout}
is split over $\Cl_K[p]$.}
\label{Table:fractions}
\end{table}

More precisely, Table~\ref{Table:fractions} lists, for the first $N_p$ imaginary
quadratic fields $K\in S_p$ of absolute discriminant $|D_K|> B_p$,
the fraction $f_p$ of $K$ for which the sequence \eqref{eq:pushout}
is split over $\Cl_K[p]$.
We started counting for absolute discriminants exceeding
$B_p$ to avoid the influence that using many very small discriminants 
may have on observing the asymptotic behavior.
Numerically, the values for $p\cdot f_p \approx 1$ in the table show that
the fraction $f_p$ is indeed close to $1/p$.

For the first three odd primes, we also
looked at the distribution of the splitting 
over the three kinds of local behavior in $K$ of the prime $p$
(split, inert or ramified)
and concluded that,
at least numerically, there is no clearly visible influence;
see Table~\ref{Table:local}.

\begin{table}[ht]
\begin{center}
\begin{tabular}{ccccccccccccc}
\toprule
$p$ && $N_p$ &&    $B_p$ && $p\cdot f_p$ && Split && Inert && Ramified \\
\midrule
3 && 300 && $10^7$ && 0.960 && 0.925 && 0.947 && 1.025 \\
5 && 500 && $10^7$ && 0.930 && 0.833 && 0.990 && 1.022 \\
7 && 700 && $10^7$ && 0.960 && 0.972 && 0.963 && 0.897 \\
\bottomrule
\end{tabular}
\bigskip
\end{center}
\caption{Splitting fractions at $p$ according to local behavior at $p$.
The first four columns are as in Table~\ref{Table:fractions}.
The remaining columns give the values of $p$ times the quantity
analogous to $f_p$, where we further limit our attention to fields
in which $p$ has the prescribed splitting behavior.}
\label{Table:local}
\end{table}

We further did a few computations that confirmed the natural hypothesis
that the splitting behaviors at different primes $p$ and $q$
that both divide the class number once
are independent of each other.

\section*{Acknowledgement}
We thank Georges Gras for spotting an inaccuracy in part (2) of
Theorem~\ref{thm:2rank1}, and for pointing out related 
results in his textbook on class field theory~\cite{\Gras}.

\bibliographystyle{hplaindoi}
\bibliography{AngelakisEtAl}
\end{document}